\documentclass{article}
\usepackage{amssymb,amsmath,amsthm}
\usepackage{graphicx,cite}

\def\qed{\hfill {\hbox{${\vcenter{\vbox{               
   \hrule height 0.4pt\hbox{\vrule width 0.4pt height 6pt
   \kern5pt\vrule width 0.4pt}\hrule height 0.4pt}}}$}}}

\def\otr{\ \overline{\triangleright}\ }
\def\utr{\ \underline{\triangleright}\ }

\newtheorem{theorem}{Theorem}
\newtheorem{lemma}[theorem]{Lemma}
\newtheorem{proposition}[theorem]{Proposition}
\newtheorem{corollary}[theorem]{Corollary}

\theoremstyle{definition}
\newtheorem{definition}{Definition}
\newtheorem{example}{Example}

\date{}

\title{\Large \textbf{Symmetric Enhancements of Involutory Virtual Birack 
Counting Invariants}}

\author{
Melinda Ho\footnote{Email: \texttt{melinda.y.t.ho@gmail.com}}
\and 
Sam Nelson\footnote{Email: \texttt{knots@esotericka.org}. Partially supported by Simons Foundation collaboration grant 316709}
}

\begin{document}
\maketitle

\begin{abstract}
We consider involutory virtual biracks with good involutions, also known as 
\textit{symmetric involutory virtual biracks}. Any good involution on an 
involutory virtual birack defines an enhancement of the counting invariant. 
We provide examples demonstrating that the enhancement is stronger than the 
unenhanced counting invariant.
\end{abstract}



\section{Introduction} 

\textit{Framed virtual knots} are equivalence classes of oriented knot diagrams with classical and virtual crossings under the equivalence relation generated 
by the blackboard framed virtual Reidemeister moves. They can be identified
with knotted solid tori whose ambient spaces are thickened orientable compact 
surfaces $\Sigma\times [0,1]$ up to stabilization moves on $\Sigma$.

\textit{Symmetric quandles,} also known as \textit{quandles with good 
involutions}, are algebraic structures related to unoriented 
knots in non-orientable thickened surfaces \cite{KO}. Symmetric quandles can 
be understood as quandles with extra structure given by a \textit{good 
involution}, an involutory map satisfying certain identities motivated by 
knot diagrams.

\textit{Biracks}, introduced in \cite{FRS0}, are algebraic structures 
related to framed oriented knots and links.
\textit{Virtual biracks} are algebraic structures generalizing \textit{virtual 
biquandles} (see \cite{KM}) with axioms motivated by the framed virtual 
Reidemeister moves. \textit{Involutory biquandles}, useful for defining
invariants of unoriented knots and links, were considered in \cite{AN}.
Each finite example of these algebraic structures defines a computable
integer-valued invariant of the related type of knots and links. In particular,
in \cite{N-BR} finite biracks are used to define an invariant of unframed knots
which reduces to the biquandle counting invariant when the birack in question 
is a biquandle.
 
In this paper we apply the symmetric quandle idea to the cases of virtual 
biracks to enhance the birack counting invariant. An \textit{enhancement} 
of an invariant $\Phi$ is another invariant $\Psi$ which determines $\Phi$
but may contain more information; for example, the quandle cocycle invariants
defined in \cite{CJKLS} are enhancements of the quandle counting invariant. 
For any finite involutory virtual birack $X$, each good involution on $X$ 
determines an enhancement of the counting invariant by partitioning the set of 
labelings of a knot or link into disjoint subsets similar to the 
birack homomorphism enhancements defined in \cite{NW}.

The paper is organized as follows. In Section \ref{V} we review unoriented
virtual knots and involutory virtual biracks. In Section \ref{SVB} we define 
good involutions of involutory virtual biracks and give some examples. In 
Section \ref{SE} we use good involutions to enhance the involutory
virtual birack counting invariant and compute some examples.
We conclude with some questions for future research in Section \ref{Q}.

\section{Virtual Knots and Involutory Virtual Biracks}\label{V}

\textit{Framed virtual knots} are equivalence classes of unoriented knot 
diagrams including \textit{classical crossings} and \textit{virtual crossings}
\[\includegraphics{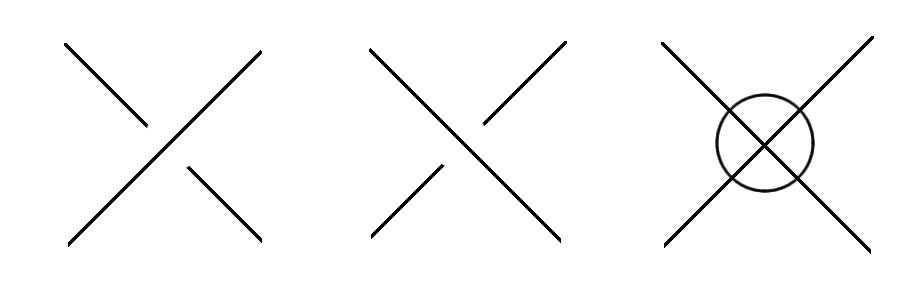}\]
under the equivalence relation generated by the \textit{framed virtual 
Reidemeister moves}:
\[\includegraphics{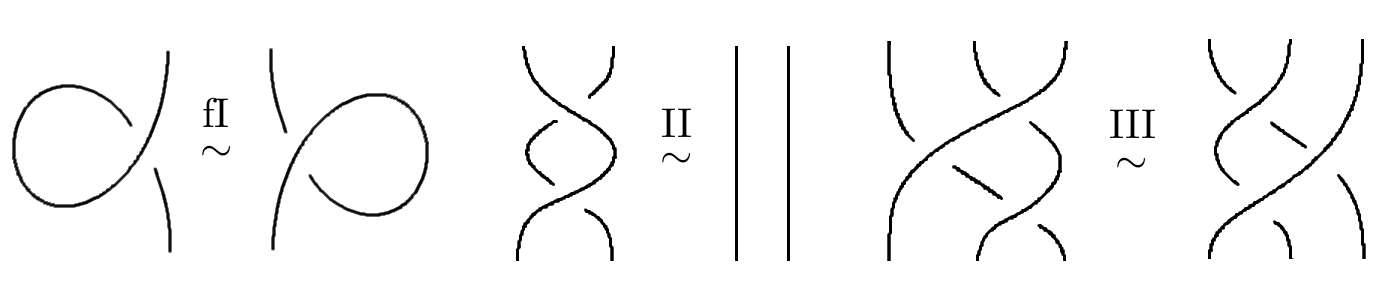}\]
\[\includegraphics{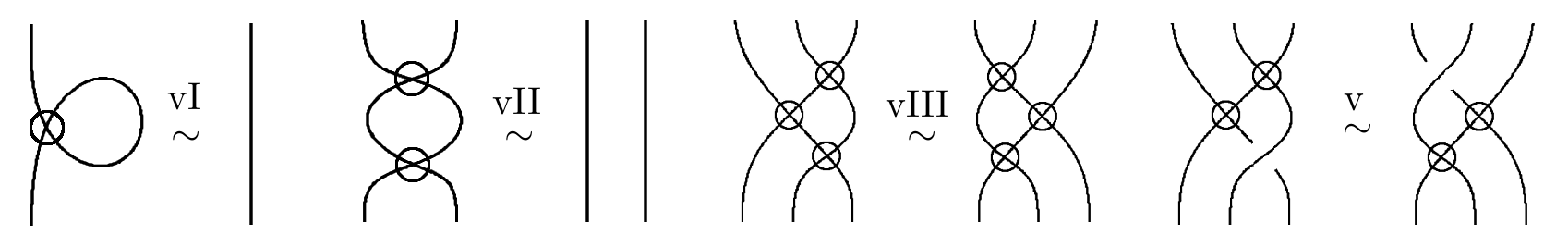}\]

Framed unoriented virtual knots can be understood as ambient isotopy classes of 
knotted solid tori in oriented thickened surfaces 
$\Sigma\times[0,1]$ up to stabilization of $\Sigma$. Classical crossings then 
represent points where the knotted solid tori cross within 
$\Sigma\times [0,1]$ while the virtual crossings are artifacts of projecting 
$\Sigma$ onto the plane of the paper.
See \cite{CKS,KK,K} for more about virtual knots.

\begin{definition}
Let $X$ be a set. An \textit{involutory virtual birack} structure on $X$ 
consists of 
three binary operations $\otr,\utr,\circledast:X\times X\to X$ and a bijection 
$\pi:X\to X$ called the \textit{kink map} satisfying 
\begin{itemize}
\item[(i)] $\pi(x)\otr x=x\utr \pi(x)$, $\pi(x\otr x)=x\utr x$ 
$\pi(x)\utr x=x\otr \pi(x)$ and $\pi(x\utr x)=x\otr x$ for all 
$x\in X$,
\item[(ii)] The maps $\alpha_y,\beta_y,v_y:X\to X$ defined by 
$\alpha_y(x)=x\otr y$, $\beta_y(x)=x\utr y$ and $v_y(x)=x\circledast y$ are 
involutions satisfying 
\[
\alpha_{y}(x)=\alpha_{\beta_x(y)}(x), \quad
\beta_{y}(x)=\beta_{\alpha_x(y)}(x), \quad \mathrm{and}\quad
v_{y}(x)=v_{v_x(y)}(x),
\]
and the maps $S:X\times X\to X\times X$ and 
$V:X\times X\to X\times X$ defined by $S(x,y)=(y\otr x,x\utr y)$ and 
$V(x,y)=(y\circledast x,x\circledast y)$ are bijections,
\item[(iii)] The \textit{exchange laws} are satisfied:
\[\begin{array}{rclc}
(x\otr y)\otr (z\otr y) & = & (x\otr z)\otr (y\utr z) & (iii.i) \\
(x\otr y)\utr (z\otr y) & = & (x\utr z)\otr (y\utr z) & (iii.ii) \\
(x\utr y)\utr (z\utr y) & = & (x\utr z)\utr (y\otr z) & (iii.iii) \\
(x\circledast y)\circledast (z\circledast y) 
& = & (x\circledast z)\circledast(y\circledast z) & (iii.iv) \\
(x\otr y)\circledast (z\circledast y) 
& = & (x\circledast z)\otr (y\circledast z) & (iii.v) \\
(x\circledast y)\circledast (z\otr y) 
& = & (x\circledast z)\circledast (y\utr z) & (iii.vi) \\
(x\utr y)\circledast (z\circledast y) 
& = & (x\circledast z)\utr (y\circledast z) & (iii.vii) \\
\end{array}\]
\end{itemize}
The exponent of $\pi$ in the symmetric group on $X$, i.e. the minimal integer 
$N\ge 1$ such that $\pi^N=\mathrm{Id}:X\to X$, is the \textit{birack 
characteristic} or \textit{birack rank} of $X$. An involutory virtual birack of 
characteristic $N=1$ is a \textit{virtual bikei}. An involutory virtual birack 
in which $x\circledast y=x$ for all $x,y\in X$ is an \textit{involutory
birack}. 
\end{definition}

\begin{example}
Any group $G$ has an involutory virtual birack structure given by 
\[x\utr y=yx^{-1}y,\quad x\otr y=x,\quad x\circledast y = x, \quad\pi(x)=x.\]
Such an involutory virtual birack is known as a \textit{core quandle}. 
See \cite{J,M} for more.
\end{example}

\begin{example}
Let $A$ be any module over $\mathbb{Z}[t, s, v]/(1-t^2,1-s^2,1-s+t-st,1-v^2)$
and define
\[x\utr y=tx+(1-st)y,\quad x\otr y= sx, \quad x\circledast y=vx, 
\quad \pi(x)=x\]
Then $A$ is a \textit{virtual Alexander bikei}. See \cite{AN, KR,KM} for more.
\end{example}

\begin{example}
Let $X$ be any set with commuting involutions $\sigma,\tau,\nu:X\to X$
and define
\[x\utr y=\sigma(x),\quad x\otr y= \tau(x), \quad x\circledast y=\nu(x), \quad 
\pi(x)=\tau^{-1}\sigma x.\]
Then $A$ is a \textit{constant action involutory virtual birack}. 
\end{example}

More generally, we can conveniently express an involutory virtual birack 
structure on a 
finite set $X=\{x_1,\dots, x_n\}$ with a matrix of three $n\times n$ blocks
encoding the operation tables of $\utr,\otr$ and $\circledast$, i.e. the 
$n\times 3n$ matrix whose $(i,j)$ entry is $x_k$ where
\[x_k=\left\{
\begin{array}{ll}
x_i\utr x_j & 1\le j\le n \\
x_i\otr x_j & n+1\le j\le 2n \\
x_i\circledast x_j & 2n+1\le 3n.
\end{array}
\right.\]
\begin{example} 
The constant action involutory virtual birack on $X=\{x_1,x_2,x_3,x_4\}=
\{1,2,3,4\}$ with
$\sigma=(12)$, $\tau=(12)(34)$ and $\nu=(34)$ has matrix
\[M=\left[\begin{array}{rrrr|rrrr|rrrr}
2 & 2 & 2 & 2 & 2 & 2 & 2 & 2 & 1 & 1 & 1 & 1\\
1 & 1 & 1 & 1 & 1 & 1 & 1 & 1 & 2 & 2 & 2 & 2 \\
3 & 3 & 3 & 3 & 4 & 4 & 4 & 4 & 4 & 4 & 4 & 4\\
4 & 4 & 4 & 4 & 3 & 3 & 3 & 3 & 3 & 3 & 3 & 3 \\
\end{array}\right].\]
\end{example}

The involutory virtual birack axioms come from the Reidemeister moves with 
semiarcs labeled according to the rules
\[\includegraphics{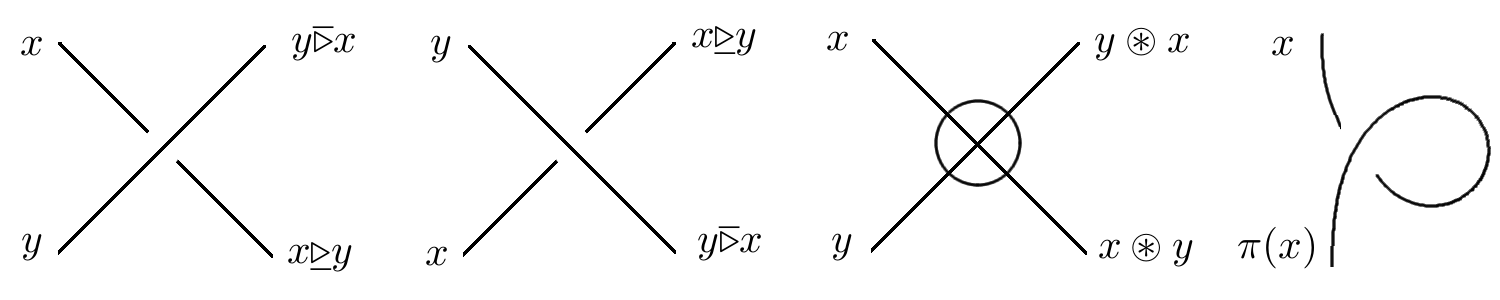}.\]
The involutory virtual birack axioms are chosen such that for every valid 
involutory birack labeling 
of a framed virtual knot before a virtual Reidemeister move, there is a unique 
corresponding labeling after the move.
The single-strand moves I and vI impose the requirements in axiom (i) 
(we depict one of the two classical cases): 
\[\includegraphics{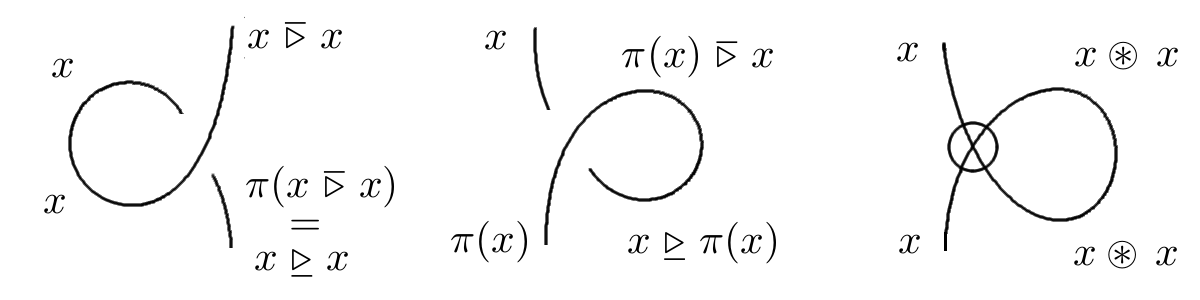}.\]
The kink map $\pi$ represents passing through a positive kink; its 
inverse $\pi^{-1}$ represents passing through a negative kink. If we define maps
$f,g:X\to X$ by $f(x)=x\otr x$ and $g(x)=x\utr x$, then we can take the first
equation above as a definition for $\pi$, namely set $\pi(x)=g(f^{-1}(x))$.

The two-strand moves II and vII require that at negative crossings the 
operations are switched and the virtual operation is the same on the top as 
on the bottom:
\[\includegraphics{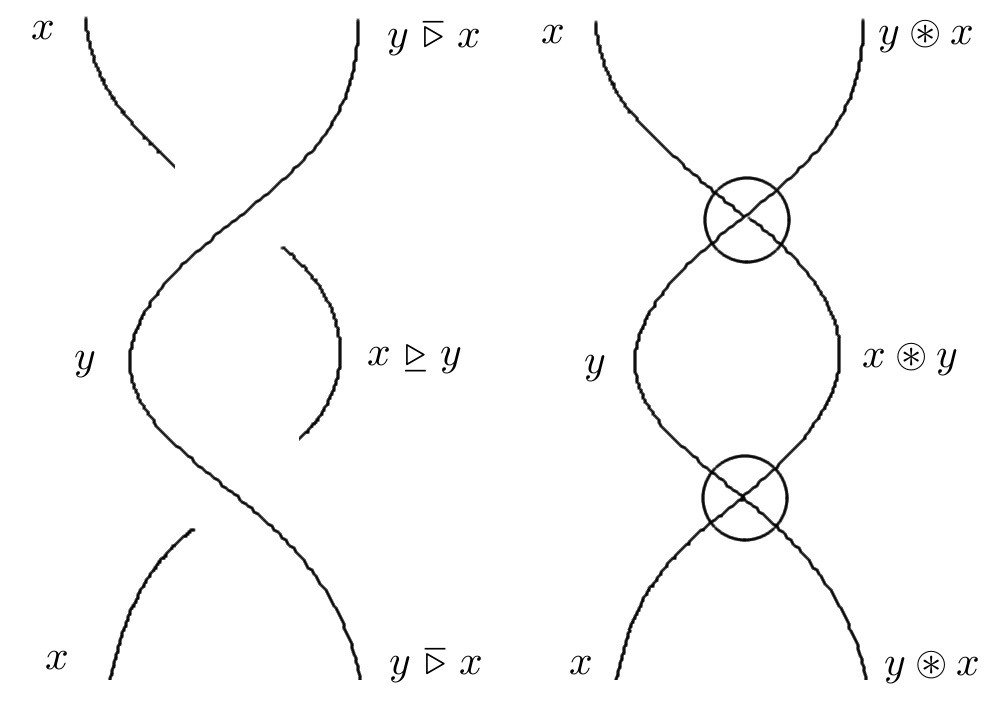}\]
The invertibility requirements are equivalent to the \textit{adjacent pairs
rule}, which says that at any crossing any pair of adjacent labels determines 
the other two labels. In particular, we need the operations to be 
right-invertible to make the labels on the right side of boundary of the 
neighborhood of the moves generic, i.e., since any pair of adjacent labels
can be considered to be ``the inputs'', we need $x\ast y$ to be an arbitrary 
element of $X$ for each $\ast\in\{\utr,\otr,\circledast\}$, which implies that 
operations are right-invertible.

The three-strand moves yield the exchange laws. We illustrate with the 
classical type III move; moves v and vIII are similar.
\[\includegraphics{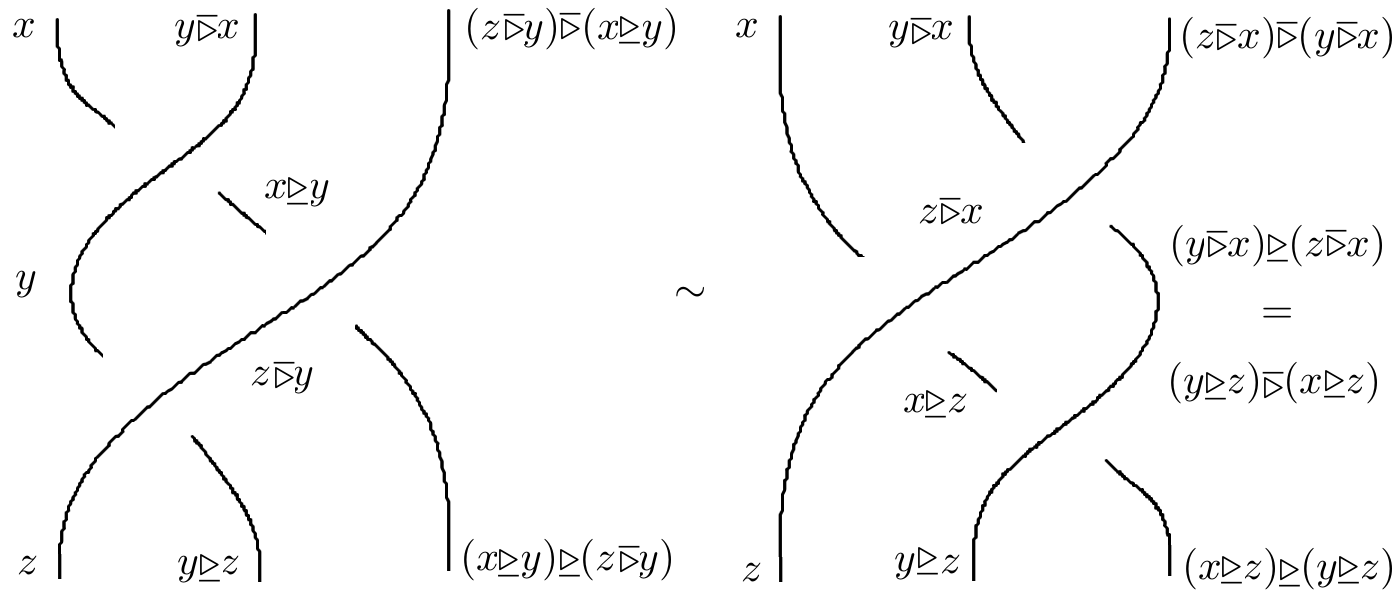}\]

By construction, we have the following standard result:

\begin{theorem}
If $X$ is an involutory virtual birack and $L$ and $L'$ are two unoriented
virtual link diagrams 
related by framed virtual Reidemeister moves, there is a one-to-one 
correspondence between the sets of $X$-labelings of $L$ and $L'$.
\end{theorem}

\begin{definition}
Let $X$ and $Y$ be involutory virtual biracks. A map $f:X\to Y$ is
a \textit{homomorphism} of involutory virtual biracks if we have
\[f(x\ast y)=f(x)\ast f(y)\]
for all $x,y\in X$ and $\ast\in\{\utr,\otr,\circledast\}$.
\end{definition}

The set of $X$-labelings of a virtual knot or link diagram can be identified 
with the set $\mathrm{Hom}(\mathcal{IVB}(L),X)$ of involutory virtual birack 
homomorphisms from the \textit{fundamental involutory virtual birack} of the
knot or link $L$, defined combinatorially as the set of equivalence classes of
virtual birack words is a set of generators corresponding to the semiarcs of 
$L$ modulo the equivalence relation generated by the involutory virtual birack
axioms and crossing relations in $L$, to $X$.

Involutory virtual biracks generalize \textit{virtual bikei}, i.e. involutory
virtual biracks with $\pi=\mathrm{Id}:X\to X$, first introduced in \cite{KM}. 
The definitions and results in this section are straightforward 
generalizations, with new notation, of results from \cite{AN,N-BR} etc.
See \cite{FJK,KR,KM} for more about oriented biracks and virtual biracks.

\section{Symmetric Virtual Biracks}\label{SVB}

We begin this section with a generalization of a definition from \cite{KO}.

\begin{definition}\label{def:good}
Let $X$ be an involutory virtual birack. An involution $\rho:X\to X$ is a 
\textit{good involution} if for all $x,y\in X$ we have
\[\rho(x) \ast y=\rho(x\ast y) \quad \mathrm{and}\quad
x\ast \rho(y)=x\ast y\]
where $\ast\in \{\utr, \otr,\circledast\}$.
An involutory virtual birack with a choice of good involution is a 
\textit{symmetric virtual birack}.
\end{definition}

We visualize elements of a symmetric virtual birack as semiarcs labels 
consisting of arrows normal to the semiarc, with $\rho(x)$ pointing in the
opposite direction of $x$. 
\[\includegraphics{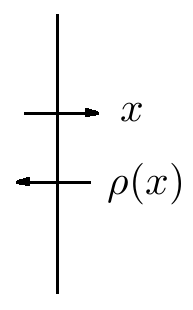}\]
We will refer to such a labeling of a link diagram $L$ with normal vectors
labeled with elements of $X$ as an \textit{arrow labeling} of $L$ by
$(X,\rho)$.

The good involution conditions then arise by picturing $x\ast y$ as 
sliding the arrow represented by $x$ through the crossing, maintaining 
normality with the strand, in the direction of $y$. 
\[\includegraphics{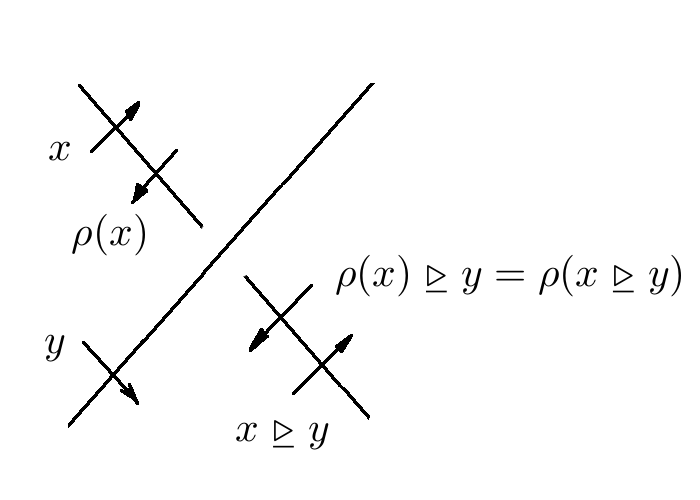}\quad
\quad \includegraphics{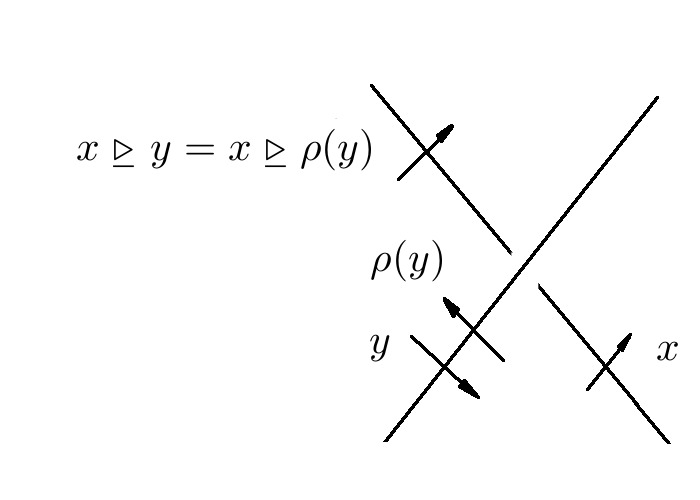}
\]

Given an involutory virtual birack, we can ask which involutions 
$\rho:X\to X$ are good.

\begin{lemma}\label{lem:pi}
Let $X$ be an involutory virtual birack. Then the kink map $\pi:X\to X$ 
satisfies the conditions
\[
\begin{array}{rclrclr}
y\otr x & = & y\otr \pi(x), & 
\pi(x\utr y) & = & \pi(x)\utr y & (+)\\
y\utr x & = & y\utr \pi(x), & 
\pi(x\otr y) & = & \pi(x)\otr y & (-)\\
y\circledast x & = & y\circledast \pi(x), & 
\pi(x\circledast y) & = & \pi(x)\circledast y & (\mathrm{v}).
\end{array}
\]
\end{lemma}

\begin{proof}
This is most easily verified using diagrams. We depict the positive crossing
case below; the negative and virtual crossing cases are similar.
\[\includegraphics{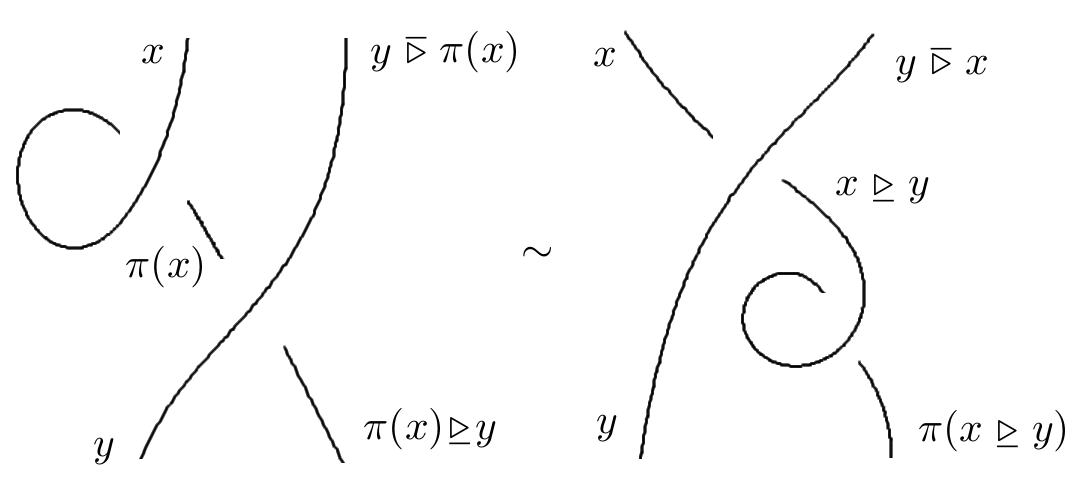}\]
\end{proof}

\begin{corollary}
Let $X$ be an involutory virtual birack. Then the kink map 
$\pi:X\to X$ is a good involution. (See also proposition 3.4 from \cite{KO}.)
\end{corollary}

The following observations will be useful in the next section.

\begin{lemma}\label{lem:hom}
The kink map $\pi:X\to X$ and any good involution $\rho:X\to X$ are 
involutory virtual birack homomorphisms.
\end{lemma}

\begin{proof}
Let $X$ be an involutory virtual birack.
For any $x,y\in X$, any $\ast\in\{\utr,\otr,\circledast\}$ 
and any good involution $\rho:X\to X$, we have
\[\pi(x\ast y) = \pi(x)\ast y=\pi(x)\ast\pi(y)\]
by Lemma \ref{lem:pi} and
\[\rho(x\ast y) = \rho(x)\ast y=\rho(x)\ast\rho(y)\]
by Definition \ref{def:good}.
\end{proof}

\begin{corollary}\label{cor:strand}
Let $X$ be an involutory virtual birack and 
$\rho:X\to X$ a good involution. Any involutory virtual birack expression
$w$ with leftmost operand $\rho(x)$ is equal to $\rho(w')$ where $w'$
is the involutory birack expression obtained from $w$ by replacing the 
leftmost operand $\rho(x)$ with $x$.
\end{corollary}

\begin{proof} 
First, consider the special cases $w=\rho(x)\ast y$ and $w=\rho(x)\ast \rho(y)$
for $\ast\in\{\utr,\otr,\circledast\}$.
By Lemma \ref{lem:hom} we have $\rho(x)\ast \rho(y)=\rho(x\ast y)$, and
by Definition \ref{def:good} we have
\[\rho(x)\ast y=\rho(x)\ast\rho(y)=\rho(x\ast y).\]
Recursively applying these two cases yields the result in the general case.
\end{proof}

\begin{example}
For instance, consider the expression 
$((\rho(x)\otr y)\utr z)\utr(x\circledast y)$. We have
\begin{eqnarray*}
((\rho(x)\otr y)\utr z)\utr(x\circledast y) 
& = & ((\rho(x)\otr \rho(y))\utr z)\utr(x\circledast y) \\
& = & ((\rho(x\otr y)\utr z)\utr(x\circledast y) \\
& = & ((\rho(x\otr y)\utr \rho(z))\utr(x\circledast y) \\
& = & (\rho((x\otr y)\utr z)\utr(x\circledast y) \\
& = & (\rho((x\otr y)\utr z)\utr\rho((x\circledast y)) \\
& = & \rho(((x\otr y)\utr z)\utr(x\circledast y)).
\end{eqnarray*}
\end{example}

\begin{lemma}\label{lem:com}
Let $X$ be an involutory virtual birack and $\rho:X\to X$ a good involution. 
Then $\pi\rho=\rho\pi$.
\end{lemma}

\begin{proof}
Recall that $\pi(x)=g(f^{-1}(x))$ where $f(x)=x\otr x$ and $g(x)=x\utr x$.
Then we have
\[\rho(g(x))=\rho(x\utr x)=\rho(x)\utr \rho(x)=g(\rho(x))\]
and similarly $\rho$ commutes with $f$, and hence with $f^{-1}.$ Then
we have 
\[\pi(\rho(x))
=g(f^{-1}(\rho(x))
=g(\rho(f^{-1}(x))
=\rho (g(f^{-1}(x))
=\rho(\pi(x))
\]
as required. We can also see this directly from the picture:
\[\includegraphics{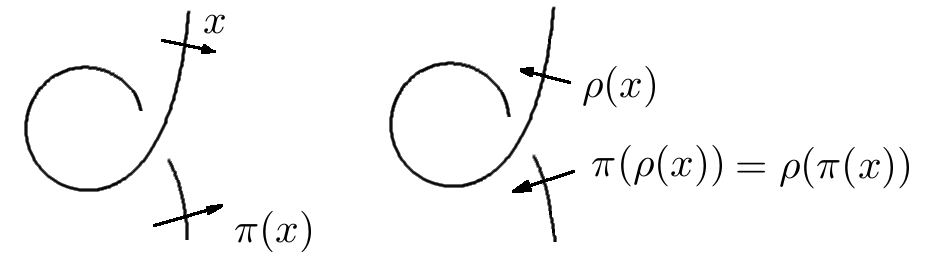}\]
\end{proof}

\section{Symmetric Enhancements}\label{SE}

Let us briefly recall the definition of the involutory birack counting 
invariant from \cite{N-BR}. Let $X$ be a finite involutory virtual birack. 
Then $\pi$ can be identified
with an element of the finite symmetric group $S_{|X|}$ and hence has finite 
order. The minimal integer $N\ge 1$ such that $\pi^N=\mathrm{Id}_X$ is
the \textit{characteristic} of $X$. Then if two framed virtual knot or link
diagrams $L$ and $L'$ are related by the framed Reidemeister moves and the 
\textit{$N$-phone cord move}
\[\includegraphics{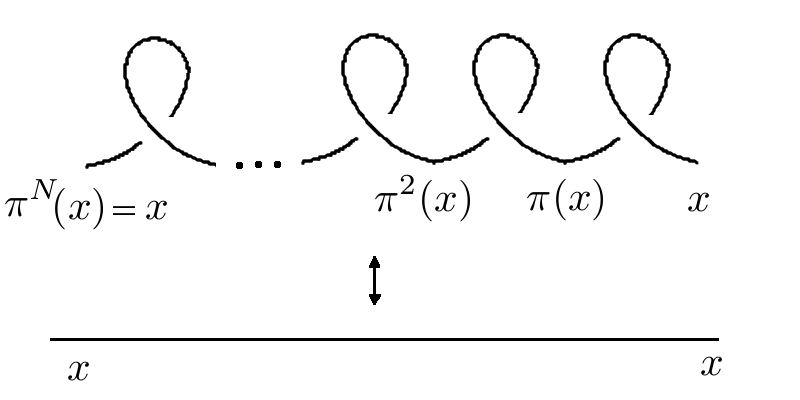}\]
then $X$-labelings of $L$ and $L'$ are in one-to-one correspondence. 
If $L$ is a virtual link of $c$ components, then the set of framings of
$c$ is in one-to-one correspondence with $\mathbb{Z}^c$, with a vector
$(w_1,\dots, w_c)\in \mathbb{Z}^c$ corresponding to a diagram of $L$
in which the $k$th component crosses itself $w_k$ times  counted 
algebraically, i.e., counting positive crossings with a $+1$ and negative
crossings with a $-1$.

Then the number of $X$-labelings of diagrams of $L$ with writhe vectors 
$\vec{w}$ and $\vec{v}$ are the same if $\vec{w}\equiv\vec{v} \ 
\mathrm{mod\ } N$. In particular, the $\mathbb{Z}^c$ lattice of $X$-labeling 
numbers is tiled by a tile which can be identified with $\mathbb{Z}_N^c$.
Then the sum over one tile of these framing numbers is an invariant of 
unframed virtual links known as the \textit{integral involutory virtual 
birack counting invariant}, denoted  
\[\Phi_X^{\mathbb{Z}}(L)=\sum_{\vec{w}\in\mathbb{Z}_N^c} |\mathcal{L}(L_{\vec{w}},X)|\]
where $L_{\vec{w}}$ is a $L$ with framing vector $\vec{w}$ and 
$\mathcal{L}(L_{\vec{w}},X)$ is the set of $X$-labelings of $L_{\vec{w}}$.

Now, suppose $X$ is an involutory virtual birack with good involution $\rho$. 
Given an arrow-labeling of a link diagram $L$, the good involution
conditions are chosen so that switching the direction of an arrow and 
applying $\rho$ to the label results in a valid arrow-labeling. 
In particular, the same $X$-labeling can
be represented by many different $(X,\rho)$ arrow labelings.

Say that two
$X$-labelings of a link $L$ are \textit{$\rho$-equivalent} if one is obtained
from the other by applying $\rho$ to a subset of the semiarc labels. That is,
two labelings are $\rho$-equivalent if for every semiarc labeled $x$ in one 
labeling, the corresponding semiarc in the other labeling is either labeled 
$x$ or $\rho(x)$. This equivalence relation partitions the sets of labelings
into disjoint subsets; we denote the quotient sets by 
$\mathcal{L}(L_{\vec{w}},X)/\rho$. 

\begin{proposition}\label{p7}
Let $X$ be an involutory virtual birack of characteristic $N$ with good
involution $\rho$. If two $X$-labelings of a link diagram are $\rho$-equivalent
before a Reidemeister or $N$-phone cord move, they are $\rho$-equivalent
after the move.
\end{proposition}

\begin{proof}
First, note that we are only considering valid $X$-labelings of diagrams; 
in particular, randomly applying $\rho$ to labels in a validly $X$-labeled
diagram need not result in a validly $X$-labeled diagram (as opposed to 
$(X,\rho)$ arrow-labeled diagram). 

The key observation is that the operator equivalence of $y\in X$ and 
$\rho(y)\in X$ implies that replacing a label $x$ with $\rho(x)$
in a validly $X$-labeled diagram only affects semiarc labels ``downstream'',
i.e., semiarcs on the same strand as $x$, and by corollary 
\ref{cor:strand} these labels are all modified by applying $\rho$. 

The type I and $N$-phone cord cases follow from lemma \ref{lem:com}.
There are three cases for the type II move:
\[\includegraphics{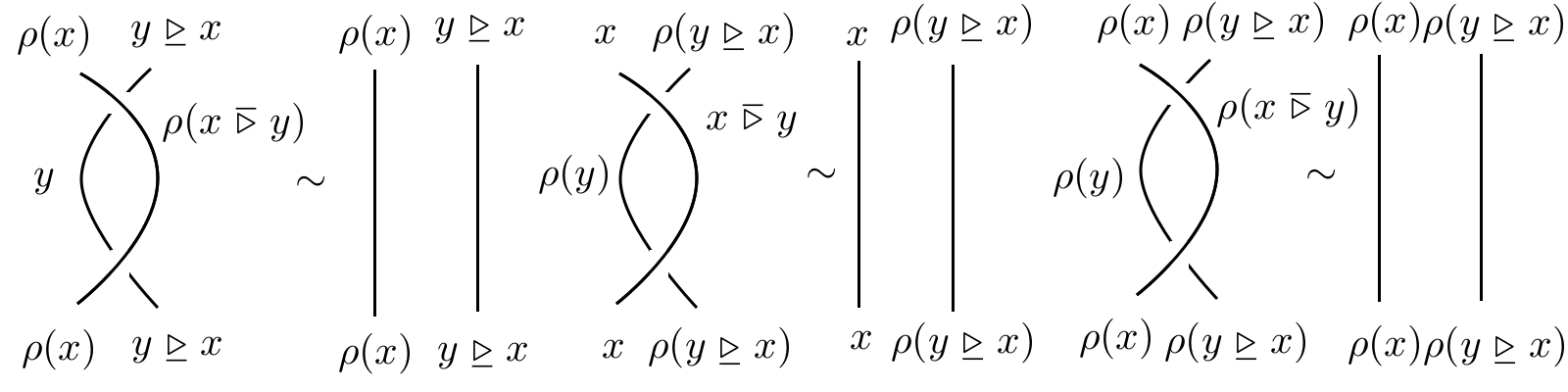}\]

For the remaining moves (III, v and vIII), we note that applying $\rho$
to one label on the edge of a move applies $\rho$ to every label on the 
same strand, so the labeling before and after the move are $\rho$-equivalent.
We illustrate one case, the others are similar.
\[\includegraphics{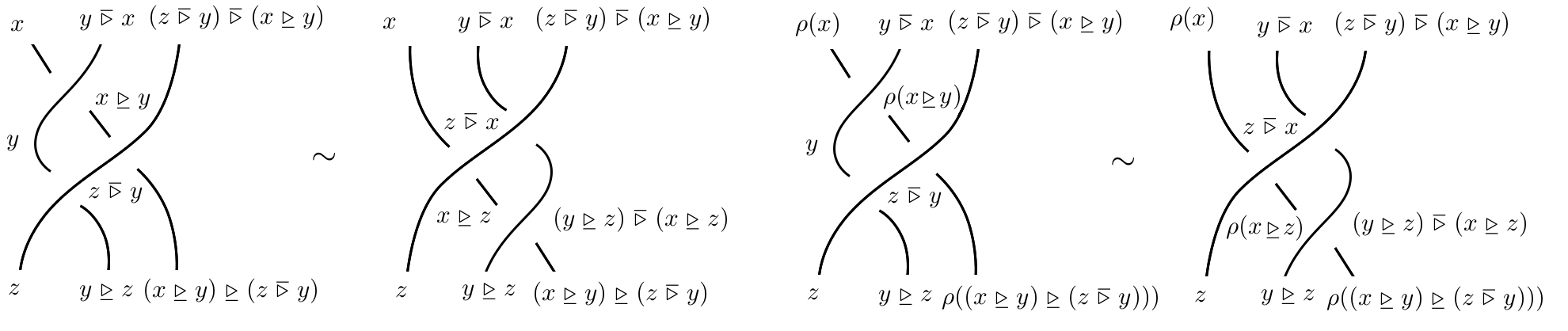}\]
\end{proof}

We can then use this extra information to enhance the counting invariant.

\begin{definition}
Let $X$ be an involutory virtual birack with good involution $\rho$. Then the 
\textit{symmetric enhancement} of the involutory virtual birack counting 
invariant is 
\[\Phi_X^{\rho}(L)=\sum_{\vec{w}\in\mathbb{Z}_N^c} 
\left(\sum_{x\in\mathcal{L}(L_{\vec{w}},X)/\rho}  u^{|x|}\right).\]
\end{definition}

As a consequence of proposition \ref{p7}, we obtain:

\begin{theorem}
Let $X$ be a finite involutory virtual birack with good involution $\rho$.
Then $\Phi_X^{\rho}(L)$ is an invariant of unframed unoriented virtual links.
\end{theorem}


If $\rho$ has no fixed points, then for every $X$-labeling of a diagram $L$
there is exactly one other $\rho$-equivalent $X$-labeling, obtained by applying
$\rho$ to every label; in this case, the enhanced invariant is equivalent
to the unenhanced invariant with
\[\Phi_X^{\rho}(L)=\frac{1}{2}\Phi_X^{\mathbb{Z}}(L) u^2.\] 
Similarly, if $\rho=\mathrm{Id}_X$ is the identity map on $X$ then we have
\[\Phi_X^{\rho}(L)=\Phi_X^{\mathbb{Z}}(L) u.\] 
If $\rho\ne \mathrm{Id}_X$ has fixed points, however, these equivalence classes 
can have various sizes and the enhanced invariant can contain more information
about $L$ than the unenhanced invariant.

\begin{example}
Let $L$ be the pictured \textit{virtual Hopf link} and let $X$ be the
involutory virtual birack with matrix
\[\left[\begin{array}{rrr|rrr|rrr}
1 & 1 & 1 & 1 & 1 & 1 & 1 & 1 & 1 \\
2 & 3 & 3 & 3 & 2 & 2 & 3 & 3 & 3 \\
3 & 2 & 2 & 2 & 3 & 3 & 2 & 2 & 2
\end{array}\right]\]
This involutory virtual birack has kink map $\pi=(23)$ and hence 
characteristic $N=2$, 
and $\rho=(23)$ is a good involution. Then there are sixteen total 
$X$-labelings of framings of $L$ over a tile of framing vectors mod 2.
In framing $\vec{w}=(0,0)$ we get contribution $u+u^2$
\[\includegraphics{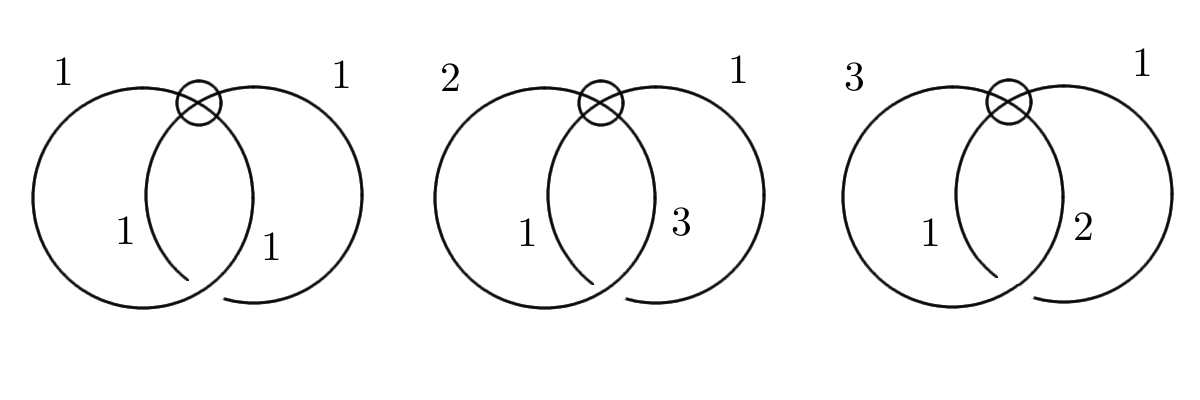},\]
in framing $\vec{w}=(1,0)$ we get contribution $u+u^4$ 
\[\includegraphics{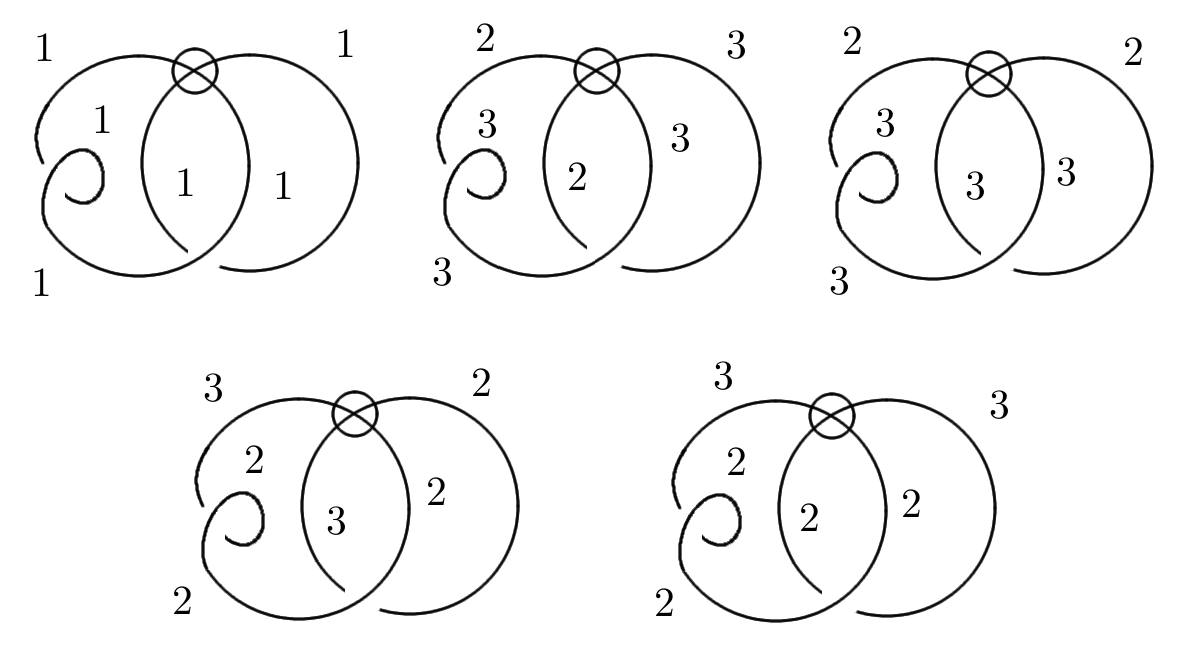},\]
in framing $\vec{w}=(0,1)$ we get contribution $u+2u^2$
\[\includegraphics{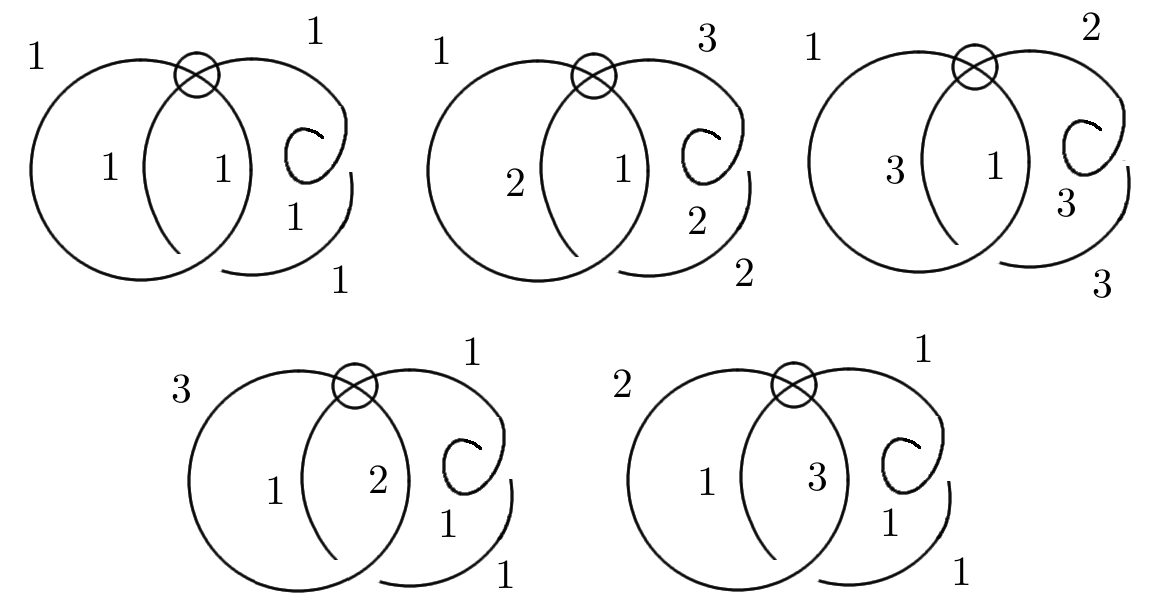},\]
and in framing $\vec{w}=(1,1)$ we get contribution $u+u^2$
\[\includegraphics{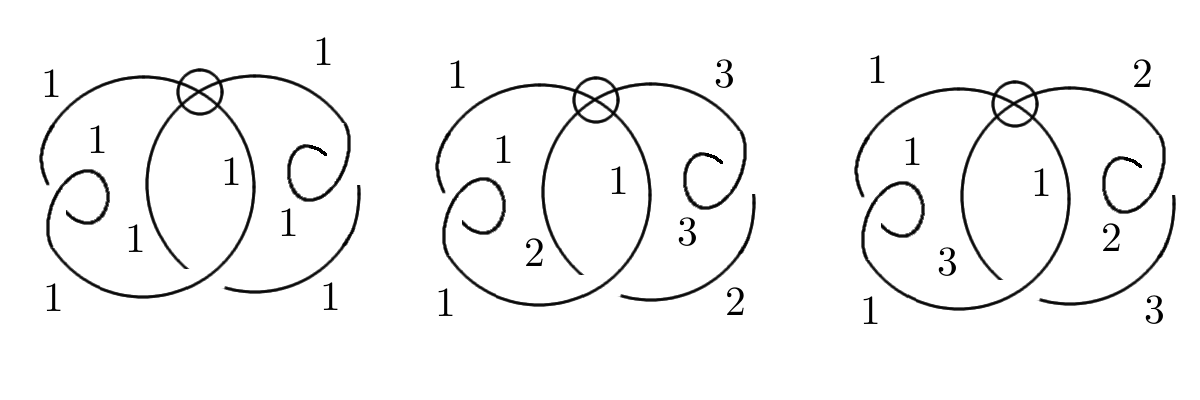}\]
Hence, the invariant value is $4u+4u^2+u^4$. 
\end{example}

Our next example demonstrates that the symmetric enhancement is a proper
enhancement, i.e., that $\Phi_X^{\mathbb{\rho}}(L)$ determines 
$\Phi_X^{\mathbb{Z}}(L)$ but is not determined by $\Phi_X^{\mathbb{Z}}(L)$.

\begin{example}
Let $X$ be the involutory virtual birack $X$ of characteristic $N=2$ with matrix
\[\left[\begin{array}{rrrr|rrrr|rrrr}
2 & 2 & 1 & 1 & 2 & 2 & 1 & 1 & 1 & 1 & 1 & 1 \\
1 & 1 & 2 & 2 & 1 & 1 & 2 & 2 & 2 & 2 & 2 & 2 \\
3 & 4 & 3 & 3 & 3 & 4 & 4 & 4 & 3 & 3 & 4 & 4 \\
4 & 3 & 4 & 4 & 4 & 3 & 3 & 3 & 4 & 4 & 3 & 3 \\
\end{array}\right]\]
and let $\rho:X\to X$ be the permutation $\rho=(34)$. (In fact, $\rho=\pi$).
Then our \texttt{python} computations reveal that the two virtual links
$L$ and $L'$ both have counting invariant value 
$\Phi_X^{\mathbb{Z}}(L)=\Phi_X^{\mathbb{Z}}(L')=120$, the links
are distinguished by their $\Phi_X^{\rho}$-values:
\[\begin{array}{cc}
\includegraphics{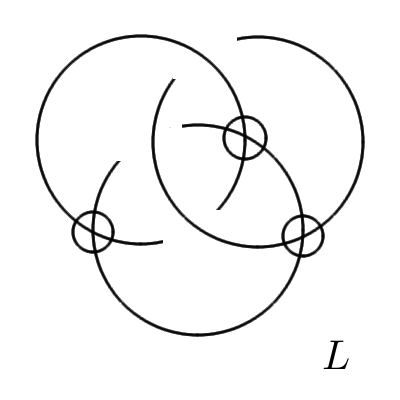} & \includegraphics{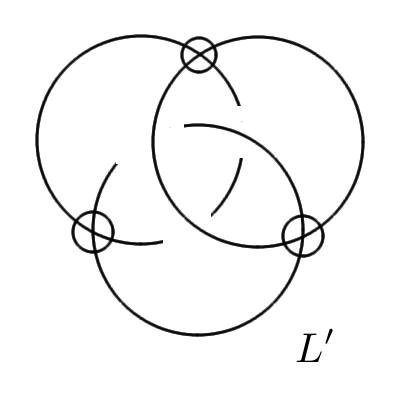} \\
\Phi_X^{\rho}(L)=u^8+12u^4+32u^2 & \Phi_X^{\rho}(L')=u^8+12u^4+64u \\
\end{array}\] 
\end{example}

\section{Questions For Future Work}\label{Q}

We close with a few questions and directions for future research.

\begin{itemize}
\item What conditions on non-involutory biracks are sufficient for
the existence of good involutions?
\item What should it mean for a map $\rho:X\to X$ with $\rho^3=\mathrm{Id}_X$
to be good?
\item In \cite{KO}, symmetric quandle cocycles are used to enhance the
quandle counting invariant. What other enhancements can be defined for 
the birack counting invariant using good involutions? 
\item What enhancements can be defined for $\Phi_X^{\rho}$?
\end{itemize}

\bibliography{mh-sn}{}

\begin{thebibliography}{10}

\bibitem{AN}
S.~Aksoy and S.~Nelson.
\newblock Bikei, involutory biracks and unoriented link invariants.
\newblock {\em J. Knot Theory Ramifications}, 21(6):1250045, 13, 2012.

\bibitem{CJKLS}
J.~S. Carter, D.~Jelsovsky, S.~Kamada, L.~Langford, and M.~Saito.
\newblock Quandle cohomology and state-sum invariants of knotted curves and
  surfaces.
\newblock {\em Trans. Amer. Math. Soc.}, 355(10):3947--3989, 2003.

\bibitem{CKS}
J.~S. Carter, S.~Kamada, and M.~Saito.
\newblock Stable equivalence of knots on surfaces and virtual knot cobordisms.
\newblock {\em J. Knot Theory Ramifications}, 11(3):311--322, 2002.
\newblock Knots 2000 Korea, Vol. 1 (Yongpyong).

\bibitem{FJK}
R.~Fenn, M.~Jordan-Santana, and L.~Kauffman.
\newblock Biquandles and virtual links.
\newblock {\em Topology Appl.}, 145(1-3):157--175, 2004.

\bibitem{FRS0}
R.~Fenn, C.~Rourke, and B.~Sanderson.
\newblock An introduction to species and the rack space.
\newblock In {\em Topics in knot theory ({E}rzurum, 1992)}, volume 399 of {\em
  NATO Adv. Sci. Inst. Ser. C Math. Phys. Sci.}, pages 33--55. Kluwer Acad.
  Publ., Dordrecht, 1993.

\bibitem{J}
D.~Joyce.
\newblock A classifying invariant of knots, the knot quandle.
\newblock {\em J. Pure Appl. Algebra}, 23(1):37--65, 1982.

\bibitem{KK}
N.~Kamada and S.~Kamada.
\newblock Abstract link diagrams and virtual knots.
\newblock {\em J. Knot Theory Ramifications}, 9(1):93--106, 2000.

\bibitem{KO}
S.~Kamada and K.~Oshiro.
\newblock Homology groups of symmetric quandles and cocycle invariants of links
  and surface-links.
\newblock {\em Trans. Amer. Math. Soc.}, 362(10):5501--5527, 2010.

\bibitem{K}
L.~H. Kauffman.
\newblock Virtual knot theory.
\newblock {\em European J. Combin.}, 20(7):663--690, 1999.

\bibitem{KM}
L.~H. Kauffman and V.~O. Manturov.
\newblock Virtual biquandles.
\newblock {\em Fund. Math.}, 188:103--146, 2005.

\bibitem{KR}
L.~H. Kauffman and D.~Radford.
\newblock Bi-oriented quantum algebras, and a generalized {A}lexander
  polynomial for virtual links.
\newblock In {\em Diagrammatic morphisms and applications ({S}an {F}rancisco,
  {CA}, 2000)}, volume 318 of {\em Contemp. Math.}, pages 113--140. Amer. Math.
  Soc., Providence, RI, 2003.

\bibitem{M}
S.~V. Matveev.
\newblock Distributive groupoids in knot theory.
\newblock {\em Mat. Sb. (N.S.)}, 119(161)(1):78--88, 160, 1982.

\bibitem{N-BR}
S.~Nelson.
\newblock Link invariants from finite biracks.
\newblock In {\em Knots in {P}oland. {III}. {P}art 1}, volume 100 of {\em
  Banach Center Publ.}, pages 197--212. Polish Acad. Sci. Inst. Math., Warsaw,
  2014.

\bibitem{NW}
S.~Nelson and E.~Watterberg.
\newblock Birack dynamical cocycles and homomorphism invariants.
\newblock {\em J. Algebra Appl.}, 12(8):1350049, 14, 2013.

\end{thebibliography}
\bibliographystyle{abbrv}

\bigskip

\noindent\textsc{Department of Mathematical Sciences\\
Claremont McKenna College \\
850 Columbia Ave. \\
Claremont, CA 91767
}

\end{document}